\documentclass[12pt,dvips]{article}

\usepackage{amsmath,amsfonts,amssymb,amsthm,graphicx,verbatim,subfig}
\usepackage[all]{xy}

\ifx\pdftexversion\undefined

\usepackage[a4paper,colorlinks,link
=black,filecolor=black,citecolor=black,urlcolor=black,pdfstartview=FitH]{hyperref}
\else

\usepackage[a4paper,colorlinks,linkcolor=black,filecolor=black,citecolor=black,urlcolor=black,pdfstartview=FitH]{hyperref}
\fi

\font\sixbb=msbm6
\font\eightbb=msbm8
\font\twelvebb=msbm10 scaled 1095
\newfam\bbfam
\textfont\bbfam=\twelvebb \scriptfont\bbfam=\eightbb
                           \scriptscriptfont\bbfam=\sixbb

\newtheorem{theorem}{\bf Theorem}[section]

\newtheorem{lemma}[theorem]{\bf Lemma}
\newtheorem{conjecture}[theorem]{\bf Conjecture}

\title{An upper bound on the number of Steiner triple systems}
\begin{document}
\author{Nathan Linial\thanks{Department of Computer Science, Hebrew University, Jerusalem 91904,
    Israel. e-mail: nati@cs.huji.ac.il~. Supported by ISF and BSF grants.}
  \and {Zur Luria\thanks{Department of Computer Science, Hebrew University, Jerusalem 91904,
    Israel. e-mail: zluria@cs.huji.ac.il~.}}
}

\date{}

\maketitle
\pagestyle{plain}

\begin{abstract}
Let $STS(n)$ denote the number of Steiner triple systems on $n$ vertices. 
Our main result is the following upper bound.
$$ STS(n) \leq \left( (1+o(1))\frac{n}{e^2} \right)^{\frac{n^2}{6}}.$$
The proof is based on the entropy method.

As a prelude to this proof we consider the number
$F(n)$ of $1$-factorizations of the complete graph on $n$ vertices.
Using the Kahn-Lov\'{a}sz theorem it can be shown that
$$ F(n) \leq \left( (1+o(1))\frac{n}{e^2} \right)^{\frac{n^2}{2}}.$$
We show how to derive this bound using the entropy method.
Both bounds are conjectured to be sharp.

\end{abstract}

\section{Introduction}
A Steiner triple system on a vertex set $V$ is a collection of triples $T \subseteq \binom{V}{3}$ such that each
pair of vertices is contained in exactly one triple from $T$. It is well known that a Steiner triple system
(STS) on $n \ge 1$ vertices exists
if and only if $n \equiv 1$ or $3$ (mod 6).

A 1-factorization of the complete graph on $n$ vertices $K_n$ is a partition of the edges of $K_n$ into $n-1$ perfect
matchings, or in other words, a proper edge coloring of $K_n$ using $n-1$ colors.
It is well known that a 1-factorization of $K_n$ exists if and only if $n$ is even.

It has been observed (e.g., \cite{Ca09}) that 1-factorizations and Steiner triple systems are
special types of Latin squares. We view a Latin square as an $n \times n \times n$ 
array $A$ with  $0-1$ entries in which each \textit{line}
has exactly one element that equals $1$. To see that this description of Latin squares is equivalent
to the usual definition, we associate to the array $A$ a matrix $L$, that is defined via $L(i,j)=k$
where $k$ is the unique index for which $A(i,j,k)=1$.
A 1-factorization is a  Latin square $A$ such that
$A(i,j,k) = 1 \Leftrightarrow A(j,i,k) = 1$ and $A(i,i,n) = 1$ for all
$i$. Thus, $L$ is a symmetric matrix in which all diagonal terms equal $n$.
A Steiner triple system is a Latin square $A$ where $A(i,j,k) = 1$ implies that $A(\sigma(i),\sigma(j),\sigma(k))=1$
for every permutation $\sigma \in S_3$ on $i,j,k$, and $A(i,i,i)=1$ for all $i$. This can also be expressed
in terms of $L$, though it's a bit more complicated to formulate.

These relations suggest that there might be deeper analogies to reveal
among Latin squares, STS's and 1-factorizations. Indeed,
we have recently proved an asymptotic upper bound on the number of Latin hypercubes
\cite{LL11}, and here we prove analogous statements for $STS(n)$ and $F(n)$.

The best previously known estimates for the number of $n$-point Steiner triple systems
are due to Richard Wilson \cite{Wi74}.

$$
\left( \frac{n}{e^2 3^{3/2}} \right)^{\frac{n^2}{6}} \le STS(n) \le \left( \frac{n}{e^{1/2}} \right)^{\frac{n^2}{6}}.
$$ 

Wilson also conjectured that, in fact, $STS(n)= \left( (1 + o(1)) \frac{n}{e^2} \right)^{\frac{n^2}{6}}$.
We show that this is an upper bound on the number of Steiner triple systems.

\begin{theorem}
 \label{thm:steiner}
$$ STS(n) \leq \left( (1 + o(1)) \frac{n}{e^2} \right)^{\frac{n^2}{6}} .$$
\end{theorem}

The Kahn-Lov\'{a}sz theorem considers a (not necessarily bipartite) graph with
degree sequence $r_1, ... ,r_n$. It shows that the number of perfect matchings
in such a graph is at most $ \prod_{i=1}^n{(r_i !)^{\frac{1}{2 r_i}}}$.
In particular a $d$-regular 
graph has at most $ (d!)^{\frac{n}{2 d}} $ perfect matchings. For a proof see Alon and 
Friedland~\cite{AF08}. These results are inspired by Br\'egman's proof~\cite{Br73} of Minc's conjecture
on the permanent. For a very recent proof of this result that uses the entropy method,
see~\cite{CR11}.

This theorem easily yields an upper bound on $F(n)$ as follows:
Choose first a perfect matching of $K_n$. The remaining edges constitute an $n-2$ regular graph in which we
again choose a 
perfect matching. We proceed to choose perfect matchings until we exhaust all of $E(K_n)$.
The theorem implies that we have at most $\left( (n-k) ! \right)^{\frac{n}{2 (n-k)}}$ choices for the $k$-th step,
so that $F(n) \le \prod_{d=1}^{n-1}{(d!)^{\frac{n}{2 d}}}$.
An application of Stirling's formula gives:
\begin{theorem}
\label{thm:factor}
$$ F(n) \leq \left( (1 + o(1)) \frac{n}{e^2} \right)^{\frac{n^2}{2}} .$$
\end{theorem}

It is an interesting question to seek lower bounds to complement
these upper bounds.
We have already mentioned Wilson's lower bound on $STS(n)$.
Cameron gave a lower bound for $F(n)$ in~\cite{Ca76}. When done with care
this argument yields 
$$ F(n) \ge \left(\frac{(1+o(1))n}{4e^2}\right)^{\frac{n^2}{2}}.$$
For the sake of completeness we repeat his argument which starts with the
inequality $F(n) \ge L(\frac{n}{2}) (F(n/2))^2$, where $L(n)$ is the number of
order-$n$ Latin squares. This inequality is shown as follows: Partition the vertex set $[n]$
into two equal parts, and select an arbitrary $1$-factor on each. It is well-known
and easy to prove that a $1$-factorization of $K_{r,r}$ is equivalent to
an order-$r$ Latin square. It follows easily from the van-der-Waerden
conjecture that $L(n) \ge (\frac{(1+o(1))n}{e^2})^{n^2}$ (see~\cite{VL+W}).
The derivation of Cameron's lower bound is a simple matter now.
We note that this argument works when $n$ is divisible by $4$.
When $n = 4r+2$ some additional care is required.

\begin{conjecture}
$$STS(n) = \left( (1 + o(1)) \frac{n}{e^2} \right)^{\frac{n^2}{6}}$$.

$$F(n) = \left( (1 + o(1)) \frac{n}{e^2} \right)^{\frac{n^2}{2}}$$.
\end{conjecture}

Our proofs are based on the entropy method, a useful tool for a variety of counting problems.
The basic idea is this: In order to estimate the size of a finite set $\cal F$, we introduce a random variable
$X$ that is uniformly distributed on the elements of $\cal F$.
Since  $H(X) = \log(|\cal{F}|)$, bounds on $H(X)$ readily translate into bounds on $|\cal{F}|$.
The bounds we derive on $H(X)$ are based on several elementary properties of the entropy function.
Namely, if a random variable takes
values in a finite set $S$ then its entropy does not exceed $\log |S|$ with equality iff the
distribution is uniform over $S$. Also, if $X$ can be expressed as $X=(Y_1,\ldots,Y_k)$,
then $H(X) = \sum_j H(Y_j|Y_1,\ldots,Y_{j-1})$. The expression $X=(Y_1,\ldots,Y_k)$ can be viewed 
as a way of gradually revealing the value of the random variable $X$. It is a key ingredient of our proofs to
randomly select the order $\prec$ in which the variables $Y_i$ are revealed and average over the resulting identities 
$H(X) = \sum_j H(Y_j|Y_i~\mbox{s.t.~}i \prec j)$.
Similar ideas can be found in the literature, but to the best of our knowledge
this method of proof is mostly due to Radhakrishnan \cite{Ra97}. 
We deviate somewhat from the standard notation
in that our logarithms are always natural, rather than binary. Formally, we should
use the notation $H_e$ for the entropy function, but to simplify matters, we stick to the standard notation $H(X)$.
We refer the reader to \cite{CT91} for a thorough discussion of entropy.
For an example of the entropy method, see
\cite{Ra97}.

In section 2, we give an entropy proof of theorem \ref{thm:factor}. Using similar methods, in section 3 we give 
an entropy proof of theorem \ref{thm:steiner}.

\section{An upper bound on 1-factorizations}
Let $n$ be an even integer, and let $X$ be a random, uniformly
chosen 1-factorization of $K_n$.
Define the random variable $X_{i,j}=X_{j,i}$ to be the color of the edge $\{i,j\}$ in $X$.
In order to analyze these random variables we first select a random ordering, denoted $\ll$, of the vertices.
Using the relation $\ll$ we introduce next a random ordering
$\prec$ of the edges as follows:
For each vertex $v$ we choose a random ordering of $E_v$, the set of edges $\{v,u\}$
where $v \ll u$. To define the ordering $\prec$,
we scan the vertices in the order $\ll$. For each vertex $v$, we
scan the edges $\{u,v\} \in E_v$ in their chosen order. Our proof proceeds by
successively revealing the colors of the edges, i.e.
the values taken by the variables $X_{i,j}$, where the edges are exposed in 
the order $\prec$. 

Given two vertices $i \neq j$, we are interested in the (random) number
of colors which are available
for the edge $\{i,j\}$, given the values taken by the $\prec$-preceding edges.
We are unable to determine this number exactly. Rather we define a random variable
$N_{i,j}$ that is an upper bound on this number.
If $j \ll i$, then the variable $X_{i,j}$ is determined by the preceding variable $X_{j,i}$,
so in this case it is natural to define $N_{i,j}=1$.
We proceed to the more interesting case where $i \ll j$. Here are two
reasons why some color may be unavailable for $X_{i,j}$.
For every vertex $t \ll i$ we already know the colors of the edges
$\{t,i\}$ and $\{t,j\}$, neither of which can be used for the edge $\{i,j\}$.
The set of colors that are ruled out for this reason is denoted $A_{i,j}$.
It is also possible that $i \ll k$ and $\{i,k\} \prec \{i,j\}$,
so that $\{i,j\}$ cannot take the color $X_{i,k}$. The set of such colors is denoted $B_{i,j}$.
Formally:

\begin{itemize}
 \item 
$A_{i,j}:= \{X_{t,i} | t \ll i\} \cup \{X_{t,j} | t \ll i\}$.
\item
$B_{i,j} := \{X_{i,k} | i \ll k \mbox{~and~}\{i,k\} \prec \{i,j\}\}$.
\end{itemize}

The set of colors that are not ruled out for the first reason is denoted:
$$ {\cal M}_{i,j} := [n-1] \smallsetminus A_{i,j}$$
and those remaining after further forbidding colors due to the second reason:
$$ {\cal N}_{i,j} := {\cal M}_{i,j} \smallsetminus B_{i,j} .$$

As mentioned, we seek to define a random variable $N_{i,j}$ that is an upper bound
on the number of possible values for $X_{i,j}$ given the $\prec$-previous edge colors.
To this end we define $N_{i,j}$ as the cardinality of ${\cal N}_{i,j}$.
As it turns out, a cruder upper bound on the number
of possible values for $X_{i,j}$ is useful as well. Namely, one that takes into account
only the colors of the edges involving vertices that $\ll$-precede $i$.
This is accomplished by the random variable $M_{i,j}$ which is defined as $|{\cal M}_{i,j}|$.

Fix an ordering $\prec$. We apply the chain rule for the entropy function and conclude that
\[
\log(F(n)) = H(X) = \sum_{(i,j)}{H(X_{i,j}| X_e : e \prec \{i,j\})} \leq
\sum_{(i,j)}{\mathbb{E}_X[\log(N_{i,j})]}.
\]

Next we take the expectation with respect to the random choice of the order $\prec$. 
\[
\log(F(n)) \leq  \mathbb{E}_{\prec}[\sum_{(i,j)}{\mathbb{E}_X[ \log(N_{i,j})]} ] = 
\sum_{(i,j)} \mathbb{E}_X [\mathbb{E}_{\prec}[ \log(N_{i,j})]].
\]
Fix a 1-factorization $X$ and a pair $i \neq j$. If $j \ll i$, then $\log(N_{i,j}) = 0$.
The probability that $i \ll j$ is $\frac12$, so that 
\[
 \log(F(n)) \leq \frac12 \sum_{(i,j)} \mathbb{E}_X[\mathbb{E}_{\prec|i \ll j}[\log(N_{i,j})]].
\]

A natural approach is to bound the expectation $\mathbb{E}_{\prec|i \ll j}[\log(M_{i,j})]$
using Jensen's inequality. As it turns out, this yields a somewhat weaker upper bound. Rather we
argue as follows:
\[
\mathbb{E}_{\prec|i \ll j}[\log(M_{i,j})] = \mathbb{E}_{\ll|i \ll j}[\log(M_{i,j})] = 
\]
\begin{equation}\label{eq1}
\mathbb{E}_{p|i \ll j}[\mathbb{E}_{\ll|p,~i \ll j}[\log(M_{i,j})]  ] \leq
\mathbb{E}_{p|i \ll j}[  \log( \mathbb{E}_{\ll|p,~i \ll j}[M_{i,j}])]
\end{equation}

For the first equality note that ${\cal M}_{i,j}$ depends only on the ordering $\ll$.
Next we condition on $p$, the position of $i$ in $\ll$ and then, finally we
resort to Jensen's inequality. 
In order to bound this expression it is necessary to understand the distribution of $p$ and 
the expectation of $M_{i,j}$ given $p$. 
\begin{lemma}
\label{dist_p}
The probability that $i$ occupies the $p$-th position in $\ll$,
given that $i \ll j$ is
\[2\frac{n-p}{n(n-1)}.\]
\end{lemma}

\begin{proof}
We are sampling uniformly
from among the $\frac{n!}{2}$ permutations in which $i\ll j$.
To specify such a permutation in which $i$ is in the $p$-th position, we must assign $j$ to 
one of the $n-p$ positions following $i$. There are $(n-2)!$ ways to order remaining elements with a total of
$(n-p) (n-2)!$ such permutations. The conclusion follows.
\end{proof}

\begin{lemma}
\label{exp_M}
$\mathbb{E}_{\ll|p,~i\ll j}[M_{i,j}] = 1 + \frac{(n-p-1)(n-p-2)}{(n-1)}$.
\end{lemma}

\begin{proof}
Now we are sampling uniformly from among the $(n-p)(n-2)!$ permutations in which $i$ is
in the $p$-th position and $i \ll j$.
If $X_{i,j}=s$, then clearly the color $s$ belongs to ${\cal M}_{i,j}$. This corresponds to the $1$ term in the lemma. 
For any other color $t \neq s$, let $a$ (resp. $b$) be the unique vertex such that $X_{i,a} = t$ (resp. $X_{j,b} = t$). Clearly,
$t \in {\cal M}_{ij}$ iff $i \ll a,b$. But
\[
 \Pr(i \ll a,b | i\mbox{~is in position~}p, i\ll j) = \frac{(n-p-1)(n-p-2)}{(n-1)(n-2)}.
\]
There are $n-2$ colors $t \neq s$ and the conclusion follows.
\end{proof}

Using lemmas \ref{dist_p} and \ref{exp_M}, we have
\[
\mathbb{E}_{p|i \ll j}[\log(\mathbb{E}_{\ll|p,~i \ll j}[M_{i,j}])] = 
\]
\[
\sum_{p=1}^{n-1}{2\frac{n-p}{n(n-1)} \log(1 + \frac{(n-p-1)(n-p-2)}{(n-1)})}=
\]
\[
= \frac{2}{n(n-1)}\sum_1^{n-2} r \log (1+ \frac{r(r-1)}{n-1}) =
\]
\[
= \frac{2}{n^2}\sum_1^{n-1} r \log (1+ \frac{r(r-1)}{n-1}) + o(1) = \frac{2}{n^2}\sum_1^{n-1} r \log (\frac{r^2}{n}) + o(1).
\]
The function $r \log (\frac{r^2}{n})$ is unimodal, and its minimum is achieved at $r=\frac{\sqrt{n}}{e}$.
Therefore $$ \sum_1^{n-1} r \log (\frac{r^2}{n}) \leq \int_0^n{u \log \left(\frac{u^2}{n}\right) du} + 2 \frac{\sqrt{n}}{e}.$$
Thus,
\[
 \mathbb{E}_{p|i \ll j}[\log(\mathbb{E}_{\ll|p,i \ll j}[M_{i,j}])] \le 
\frac{2}{n^2}\int_0^n{u \log \left(\frac{u^2}{n}\right) du} + o(1)
\]
\begin{equation}\label{eq2}
\log n - 1 + o(1) .
\end{equation}

We next proceed to consider colors that are ruled out due to variables that correspond to edges in $E_i$. 
An edge $\{i,k\}$ may rule out additional colors if $X_{i,k} \in {\cal M}_{i,j}$.
There are $M_{i,j} - 1$ such edges, one for each color in ${\cal M}_{i,j} \smallsetminus X_{i,j}$. 
Consequently, we are only interested in counting such edges that $\prec$-precede $\{i,j\}$.

\[
\mathbb{E}_{\prec|i \ll j}[\log(N_{i,j})] =  \mathbb{E}_{\prec|M_{i,j},i \ll j}[\log(N_{i,j})]=
\]
\[
\sum_l \Pr(M_{i,j}=l) \mathbb{E}_{\prec|M_{i,j} = l,i \ll j}[\log(N_{i,j})]= 
\]
\[
\sum_l \Pr(M_{i,j}=l) \frac{\log(l!)}{l} = \sum_l \Pr(M_{i,j}=l) (\log l -1 +o(1))=
\]
\[
\mathbb{E}_{\prec|i \ll j}[\log(M_{i,j})] - 1 + o(1) \le \log n - 2 + o(1).
\]

We used the fact that given $M_{i,j} = m$, the number $N_{i,j}$ of possible values
for the random variable $X_{i,j}$ is uniformly distributed between $1$ and $m$.
In the final step we used Equations~\ref{eq1} and~\ref{eq2}.

Consequently,
$$ \log(F(n)) \leq \frac{1}{2} \sum_{(i,j)}( \log n - 2 + o(1)) = \binom{n}{2}(\log n - 2 + o(1))$$ 
which yields the bound
$$ F(n) \leq \left((1 + o(1)) \frac{n}{e^2} \right)^{\frac{n^2}{2}} .$$

\section{An upper bound on the number of Steiner triple systems}

The ideas here are similar, but the details are different.

Let $X$ be a uniformly chosen random Steiner triple system on $n$ vertices. Define $X_{i,j} = X_{j,i}$ 
to be the unique vertex $k$ such that $\{i,j,k\}$ is a triple in $X$. As before, we
define next a random ordering $\ll$ on the vertices and a random ordering $\prec$ of the edges. 

Fix a Steiner triple system $X$, orderings $\ll$ and $\prec$ and a pair of vertices $i \ne j$. Let $X_{i,j} = k$.
We want to define a random variable $N_{i,j}$ that is an upper bound on the number
of vertices that are available for $X_{i,j}$, given the values of the preceding variables. Let $F_{i,j}$
denote the event that $i \ll j,k$ and $\{i,j\} \prec \{i,k\}$. Clearly, 
$\Pr(F_{i,j})=\frac16$. If $F_{i,j}$ doesn't occur, then $X_{i,j}$
is uniquely determined by the preceding variables, so in this case we define $N_{i,j}$ to be $1$.

Let $t \ne X_{i,j}$ be a vertex.
We consider two classes of reasons
for which $t$ may be ruled out as the value of $X_{i,j}$ given the previously revealed choices.
The first is the union of the following three events: $t \ll i$, $X_{i,t} \ll i$ and $X_{j,t} \ll i$.
Namely, the variables corresponding to vertices that $\ll$-precede $i$ reveal a
triple that includes $t$ and either $i$ or $j$, so
that $\{i,j,t\}$ cannot be a triple in $X$.
The second possibility is the union of the events 
$\{i, X_{i,t}\} \prec \{ i,j \}$ and $\{ i,t \} \prec \{ i,j \}$, where
the revealed triple $\{i, t, X_{i,t}\}$ rules out the
possibility that $\{i,j,t\}$ is in $X$. 
 
We define the set of vertices which are ruled out for $X_{i,j}$ due to the first reason:
\[
A_{i,j}:=\{t | t \ll i \text{~or~} X_{i,t} \ll i \text{~or~} X_{j,t} \ll i\}.
\]

Among the remaining vertices we consider those that are unavailable due to the second reason

\[
B_{i,j}:=\{t \not\in A_{i,j} | \{i,t\} \prec \{i,j\} \text{~or~} \{i,X_{i,t}\} \prec  \{i,j\}\}.
\]

Further,
$$ {\cal M}_{i,j} := (V \smallsetminus \{ i,j \}) \smallsetminus A_{i,j} .$$
$$ {\cal N}_{i,j} := {\cal M}_{i,j} \smallsetminus B_{i,j} .$$

As before we define $N_{i,j}$ as the cardinality
of ${\cal N}_{i,j}$. Also, let $M_{i,j}:=|{\cal M}_{i,j}|$.

The random variable $M_{i,j}$ gives an upper bound on the number of values that are still available 
for $X_{i,j}$ given the values of the random variables that involve vertices that $\ll$-precede $i$.
Likewise, $N_{i,j}$ is an upper bound on the number of possible values for $X_{i,j}$ when all $\prec$-preceding
choices are known. 

For a given ordering $\prec$ we derive:
\[
\log(STS(n)) = H(X) = \sum_{(i,j)}{H(X_{i,j}| X_e : e \prec \{i,j\})} \leq
\sum_{(i,j)}{\mathbb{E}_X[\log(N_{i,j})]}.
\]

We take the expectation over the random choice of $\prec$ to obtain 
$$\log(STS(n)) \leq \sum_{(i,j)}{\mathbb{E}_X[  \mathbb{E}_{\prec}[\log(N_{i,j})]  ]}.$$

Let us fix $X$ and a pair $i \ne j$ and turn to bound $\mathbb{E}_{\prec}[\log(N_{i,j})]$.
With probability $\frac56$ there holds $\log(N_{i,j})=0$, so that
\[
 \mathbb{E}_{\prec}[\log(N_{i,j})] = \Pr(F_{i,j}) \mathbb{E}_{\prec|F_{i,j}}[\log(N_{i,j})] =
 \frac{1}{6}\mathbb{E}_{\prec|F_{i,j}}[\log(N_{i,j})].
\]

Clearly, ${\cal M}_{i,j}$ depends only on the ordering $\ll$.
If $p$ is the position of $i$ in $\ll$, then
\[
\mathbb{E}_{\prec|F_{i,j}}[\log(M_{i,j})] = \mathbb{E}_{\ll|F_{i,j}}[\log(M_{i,j})]=
\]
\begin{equation}\label{eq3}
\mathbb{E}_{p|F_{i,j}}[ \mathbb{E}_{\ll|p,F_{i,j}}[\log(M_{i,j})]  ] \leq
\mathbb{E}_{p|F_{i,j}}[  \log( \mathbb{E}_{\ll|p,F_{i,j}}[M_{i,j}] )  ] .
\end{equation} 
The last inequality follows from Jensen's inequality.
We next analyze the distribution of $p$ and the expectation of $M_{i,j}$ given $p$. 
In the following lemmas we denote $X_{i,j}$ by $k$.
\begin{lemma}\label{dist_p_2}
The probability that $i$ occupies the $p$-th position in $\ll$, given $F_{i,j}$, is 
$$ 3\frac{(n-p)(n-p-1)}{n(n-1)(n-2)}. $$
\end{lemma}

\begin{proof}
We are sampling $\ll$ uniformly from among the $\frac{n!}{3}$
permutations in which $i$ precedes $j$ and $k$. To specify
such a permutation in which $i$ is in the $p$-th position we place $j$ in any
of the $n-p$ positions following $i$, and then place $k$ in one of the $n-p-1$ remaining positions following $i$. 
The remaining vertices can be ordered in $(n-3)!$ ways for a total of $(n-p)(n-p-1) (n-3)!$ such permutations. 
The conclusion follows.
\end{proof}

\begin{lemma}\label{exp_M_2}
 $\mathbb{E}_{\ll|p,F_{i,j}}[M_{i,j}] = 1 + \frac{(n-p-2)(n-p-3)(n-p-4)}{(n-4)(n-5)} .$
\end{lemma}

\begin{proof}
Now we are sampling uniformly from the set of orderings in which $i \ll j,k$ where $i$ is in the $p$-th position.
Clearly $k \in {\cal M}_{i,j}$. This corresponds to the 1 term. If $t \in V \smallsetminus \{ i,j,k \}$, 
let $a$ (resp. $b$) be the unique vertex such that $X_{i,a} = t$ (resp. $X_{i,b} = t$ ). 
The vertex $t$ forms a triple with $i$ and $a$, and a triple with $j$ and $b$. 
If an edge from either of these triples is exposed before $\{i,j\}$, 
then $t$ is ruled out for $X_{i,j}$. Note that $t \in {\cal M}_{i,j}$ iff $i \ll a, b ,t$. 

But
\[
\Pr(i \ll a,b,t|i\mbox{~is in position~}p, i\ll j,k) = 
\]
\[
\frac{(n-p-2)(n-p-3)(n-p-4)}{(n-3)(n-4)(n-5)}.
\]
There are $n-3$ such vertices $t$, and the conclusion follows. 

\end{proof}

Using lemmas \ref{dist_p_2} and \ref{exp_M_2}, we have
$$\mathbb{E}_{p|F_{i,j}}[\log( \mathbb{E}_{\ll|p,F_{i,j}}[M_{i,j}])] = $$
$$\sum_{p=1}^{n-2}{3\frac{(n-p)(n-p-1)}{n(n-1)(n-2)} \log(1 + \frac{(n-p-2)(n-p-3)(n-p-4)}{(n-4)(n-5)})} = $$
$$\frac{3}{n(n-1)(n-2)}\sum_{r=2}^{n-1}{r(r-1) \log(1 + \frac{(r-2)(r-3)(r-4)}{(n-4)(n-5)})}.$$

As in the previous section, the next step is to collect
together lower order terms and obtain
\[
  \frac{3}{n^3}\int_0^n{u^2\log \left(\frac{u^3}{n^2}\right) dx} + o(1) = \log n - 1 + o(1) .
\]

Together with \ref{eq3}, this implies that 
\begin{equation}\label{eq4}
 \mathbb{E}_{\prec|F_{i,j}}[\log(M_{i,j})] = \log n - 1 + o(1).
\end{equation}
We next show that for every $1\leq l \leq n$,
\[
 \mathbb{E}_{\prec|F_{i,j},M_{i,j}=l}[\log(N_{i,j})] = \log l - 1 + o(1).
\]

Let $q$ be the position taken by $\{i,j\}$ in the uniformly chosen random ordering of the edges in $E_i$,
and let $m = |E_i|$.
Again, by Jensen
\[
\mathbb{E}_{\prec|F_{i,j},M_{i,j}=l}[\log(N_{i,j})] \leq  
\mathbb{E}_{q|F_{i,j},M_{i,j}=l}[\log(\mathbb{E}_{\prec|q,F_{i,j},M_{i,j}=l}[N_{i,j}])].
\]

The following two lemmas describe the distribution of $q$ and the expectation of $N_{i,j}$ given $q$.
We maintain the notation that $k$ is the vertex $X_{i,j}$.
\begin{lemma}
The probability that $\{i,j\}$ occupies the $q$-th position in the ordering of $E_i$, given $F_{i,j}$, is 
$$ 2\frac{(m-q)}{m(m-1)}. $$
\end{lemma}

\begin{proof}
There are $\frac{m!}{2}$
orderings of $E_i$ in which $\{i,j\}$ precedes $\{i,k\}$. There are $m-q$ possible positions for $\{i,k\}$
following $\{i,j\}$. The conclusion follows.
\end{proof}

\begin{lemma}
 $\mathbb{E}_{\prec|q,F_{i,j},M_{i,j}=l}[N_{i,j}] = 1 +  \frac{(m-q-1)(m-q-2)}{(m-2)(m-3)}(l-1) .$
\end{lemma}

\begin{proof}
Now we are sampling uniformly from the set of orderings of $E_i$ in which $\{i,j\}$ precedes $\{i,k\}$ and $\{i,j\}$ is in 
the $q$-th position. For each vertex $v$, we determine the probability that $v \in {\cal N}_{i,j}$, and then use
the linearity of the expectation to obtain the result. We consider only vertices in ${\cal M}_{i,j}$.
   
Clearly, $k \in {\cal N}_{i,j}$. This corresponds to the 1 term. If $t \in {\cal M}_{i,j} \smallsetminus \{ k \}$,
then $t \in {\cal N}_{i,j}$ iff 
$\{i,j\} \prec \{i,a\} ,\{i,t\}$, where $a$ is the unique vertex such that $X_{i,a} = t$.
But
\[
\Pr(\{i,j\} \prec \{i,a\} ,\{i,t\}|\{i,j\}\mbox{~is in position~}q, \{i,j\}\prec \{i,k\}) = 
\]
\[
 \frac{(m-q-1)(m-q-2)}{(m-2)(m-3)}.
\]

There are $l-1$ such vertices $t$, and the conclusion follows. 

\end{proof}

Therefore, 
\[ \mathbb{E}_{\prec|F_{i,j},M_{i,j}=l}[\log(N_{i,j})] \leq \]
\[\sum_{q=1}^{m-1}{\frac{2(m - q)}{m(m - 1)} \log(1 + \frac{(m-q-1)(m-q-2)}{(m-2)(m-3)} (l-1))} = \]
\[ \frac{2}{m(m-1)}\sum_{r=1}^{m-1}{r\log(1 + \frac{(r-1)(r-2)}{(m-2)(m-3)} (l-1))} .\]
As above, this is equal to
\[ \frac{2}{m^2}\int_0^m{u\log \left(\frac{u^2}{m^2} l \right) du} + o(1) = \log l - 1 + o(1).\]

Putting all of this together, we have
\[ \mathbb{E}_{\prec|F_{i,j}}[\log(N_{i,j})] = 
\sum_{l=1}^n{ \Pr_{\prec|F_{i,j}}(M_{i,j}=l) \mathbb{E}_{\prec|F_{i,j},M_{i,j}=l}[\log(N_{i,j})] } \leq \] 
\[ \sum_{l=1}^n{ \Pr_{\prec|F_{i,j}}(M_{i,j}=l) (\log(l) - 1 + o(1)) } = \mathbb{E}_{\prec|F_{i,j}}[ \log(M_{i,j})] -1 + o(1) \leq\]
\[ \log n-2+o(1) \]

Thus, $$ \log(STS(n)) \leq \frac16 \sum_{( i,j )}{(\log n - 2 + o(1))} = $$
$$ \frac{\binom{n}{2}}{3}(\log n - 2 + o(1)) ,$$ 
which yields the bound 
$ STS(n) \leq \left((1 + o(1)) \frac{n}{e^2} \right)^{\frac{n^2}{6}} $
as claimed.

\end{document}